\documentclass{article}
\usepackage[utf8]{inputenc}

\usepackage[utf8]{inputenc}
\usepackage[a4paper, margin=1in]{geometry}

\usepackage{amsmath}
\usepackage{amssymb}
\usepackage{amsthm}
\usepackage{mathtools}
\usepackage{bbm}

\usepackage[capitalise]{cleveref}

\usepackage{enumitem}

\usepackage{titlesec}

\usepackage{xcolor}

\usepackage{etoolbox} 

\usepackage{tikz}
\usetikzlibrary{decorations.pathreplacing,calligraphy}

\newtheorem{theorem}{Theorem}[section]

\newtheorem{lemma}[theorem]{Lemma}
\newtheorem{proposition}[theorem]{Proposition}
\newtheorem{definition}[theorem]{Definition}

\newtheorem{claim}[theorem]{Claim}

\newtheorem{setup}[theorem]{Setup}

\AtBeginEnvironment{theorem}{\setlist[enumerate]{label=\upshape(\roman*)}}
\AtBeginEnvironment{definition}{\setlist[enumerate]{label=\upshape(\roman*)}}
\AtBeginEnvironment{proposition}{\setlist[enumerate]{label=\upshape(\roman*)}}
\AtBeginEnvironment{lemma}{\setlist[enumerate]{label=\upshape(\roman*)}}
\AtBeginEnvironment{corollary}{\setlist[enumerate]{label=\upshape(\roman*)}}

\numberwithin{equation}{section}

\newcommand{\N}{\mathbb{N}}         
\newcommand{\cH}{\mathcal{H}}       
\newcommand{\E}{\mathbb{E}}         
\newcommand{\K}{\mathcal{K}}        
\newcommand{\F}{\mathcal{F}}        
\newcommand{\Z}{\mathbb{Z}}         
\newcommand{\eps}{\varepsilon}      

\newcommand{\ceil}[1]{\left\lceil{#1}\right\rceil}      
\newcommand{\floor}[1]{\left\lfloor{#1}\right\rfloor}   
\newcommand{\size}[1]{\left\lvert{#1}\right\rvert}      

\newcommand{\boldx}{\mathbf{x}}                 
\newcommand{\boldM}{\mathbf{M}}                 

\newcommand{\lb}{\left(}        
\newcommand{\rb}{\right)}       
\newcommand{\lsb}{\left[}       
\newcommand{\rsb}{\right]}      
\newcommand{\lcb}{\left\{}      
\newcommand{\rcb}{\right\}}     

\DeclareMathOperator{\Forb}{Forb}       
\DeclareMathOperator{\dist}{dist}       
\DeclareMathOperator{\ed}{ed}           

\DeclareMathOperator{\VW}{VW}   
\DeclareMathOperator{\VB}{VB}   
\DeclareMathOperator{\EW}{EW}   
\DeclareMathOperator{\EB}{EB}   
\DeclareMathOperator{\EG}{EG}   

\title{Removing induced powers of cycles from a graph via fewest edits}
\author{Amarja Kathapurkar\thanks{University of Birmingham, Birmingham, B15 2TT, UK. Email: \tt{a.kathapurkar@bham.ac.uk.}} , Richard Mycroft\thanks{University of Birmingham, Birmingham, B15 2TT, UK. Email: \tt{r.mycroft@bham.ac.uk.} Research supported by EPSRC Research grant EP/R034389/1.} }
\date{}

\begin{document}

\maketitle

\begin{abstract}
What is the minimum proportion of edges which must be added to or removed from a graph of density $p$ to eliminate all induced cycles of length $h$? The maximum of this quantity over all graphs of density $p$ is measured by the edit distance function, $\text{ed}_{\text{Forb}(C_h)}(p)$, a function which provides a natural metric between graphs and hereditary properties.

Martin determined $\text{ed}_{\text{Forb}(C_h)}(p)$ for all $p \in [0,1]$ when $h \in \{3, \ldots, 9\}$ and determined $\text{ed}_{\text{Forb}(C_{10})}(p)$ for $p \in [1/7, 1]$. Peck determined $\text{ed}_{\text{Forb}(C_h)}(p)$ for all $p \in [0,1]$ for odd cycles, and for $p \in [ 1/\lceil h/3 \rceil, 1]$ for even cycles. In this paper, we fully determine the edit distance function for $C_{10}$ and $C_{12}$. Furthermore, we improve on the result of Peck for even cycles, by determining $\text{ed}_{\text{Forb}(C_h)}(p)$ for all $p \in [p_0, 1/\lceil h/3 \rceil ]$, where $p_0 \leq c/h^2$ for a constant $c$. More generally, if $C_h^t$ is the $t$-th power of the cycle $C_h$, we determine $\text{ed}_{\text{Forb}(C_h^t)}(p)$ for all $p \geq p_0$ in the case when $(t+1) \mid h$, thus improving on earlier work of Berikkyzy, Martin and Peck.
\end{abstract}

\section{Introduction}\label{sec:introduction}

For a fixed constant $p \in (0,1)$, what is the expected number of edges (expressed as a proportion of $\binom{n}{2}$) that must be added or removed from the binomial random graph $G(n, p)$ to eliminate all induced copies of $C_h$, the cycle of length $h$? To answer this question, we might consider the following strategies.
\begin{enumerate}
    \item We could arbitrarily partition the vertex set of $G(n, p)$ into $\lceil h/2 \rceil - 1$ sets, then add all edges which have both endvertices within the same one of these sets. By the Pigeonhole Principle any induced copy of $C_h$ within the resulting graph must have three vertices within the same set, meaning that these vertices form a triangle, giving a contradiction for $h \geq 4$. By convexity, the expected number of edge-additions needed for this is minimised when the sizes of the sets of the partition are made as equal as possible, and this expected minimum number of edits is essentially $\frac{1-p}{\lceil h/2 \rceil - 1} \binom{n}{2}$. 
    \item We could arbitrarily partition the vertex set of $G(n, p)$ into sets $W_1, B_1, \dots, B_{(\lceil h/3 \rceil - 1)}$, then delete all edges with both endvertices in $W_1$ whilst, for each $i$, adding all edges with both endvertices in $B_i$. Again the resulting graph cannot contain an induced copy of $C_h$. Indeed, each set $B_i$ must contain either one or two consecutive vertices of such a copy, and since there are fewer than $h/3$ sets $B_i$ it follows that $W_1$ must also contain two adjacent vertices of the cycle, a contradiction. Solving the appropriate quadratic program shows that the number of edges edited is minimised by taking $W_1$ to have size $(1- p) n$ and the sizes of the sets $B_i$ to be $\frac{p}{\ceil{h/3}-1}n$, and hence that the minimum number of edits is essentially $\frac{p(1-p)}{1 + (\lceil\frac{h}{3}\rceil -2) p}\binom{n}{2}$.
    \item Finally, if $h$ is odd, then we can arbitrarily partition the vertex set of $G(n, p)$ into two sets $A$ and $B$, and delete all edges with both endvertices in the same part. The resulting graph is bipartite and so contains no copies of $C_h$ (induced or otherwise), and the number of edges deleted is essentially~$\frac{p}{2}\binom{n}{2}$.
\end{enumerate}
These three strategies are illustrated in Figure~\ref{fig:Gamma_Forb_C_h}, where white vertices indicate sets in which all edges should be removed, and black vertices represent sets where all non-edges should be added as edges.

From these strategies we obtain an upper bound on the expected number of edge-edits needed to remove all induced copies of $C_h$ from $G(n, p)$ of essentially
\begin{align*}
    \begin{cases}
  \min \lcb \frac{1-p}{\lceil h/2 \rceil - 1} \binom{n}{2}, \frac{p(1-p)}{1 + (\lceil\frac{h}{3}\rceil -2) p}\binom{n}{2} \rcb & \mbox{if $h$ is even, and} \\
  \min \lcb \frac{1-p}{\lceil h/2 \rceil - 1} \binom{n}{2}, \frac{p(1-p)}{1 + (\lceil\frac{h}{3}\rceil -2) p}\binom{n}{2}, \frac{p}{2}\binom{n}{2}\rcb & \mbox{if $h$ is odd}.
\end{cases}
\end{align*}
Martin~\cite{martin2013edit} showed that for all $h \in \{3, \ldots, 9\}$ and $p \in [0,1]$, one of these strategies  is always `best possible', that is, the upper bound obtained above is matched by the lower bound and is in fact the correct answer. Here, the value of $p$ determines which function minimises the terms above, and therefore different strategies will be best possible depending on the value of $p$. Martin also showed that when $h = 10$, these strategies are best possible provided $p \in [1/7,1]$. Peck \cite{peck2013edit} showed that this is best possible for all odd $h \geq 3$ and $p \in [0,1]$. Peck also showed that if $h$ is even, then these strategies are best possible for $p \in [1/\ceil{\frac{h}{3}}, 1]$.

In this paper, we add to this picture, showing that in fact, for $h =10$, these strategies are also best possible for all $p \in [0,1/7]$. We also show that for $h=12$, these strategies are best possible for $p \in [0,1/4]$. Thus, we complete the picture for $h \in \{10, 12\}$. Furthermore, we show that for all even $h \geq 12$, there is a constant $p_0=p_0(h)$ such that for all $p \in [p_0, 1/\ceil{h/3}]$, these strategies are best possible. We remark that the constant $p_0$ is a function of $h$, and is significant because when $h=10$, the function gives exactly the value $1/7$. We give more formal statements of all of the above in \cref{subsec:main_results}. In fact, we prove a more general result where rather than eliminating induced cycles of some fixed length $h$, we are eliminating induced copies of the $t$-th power of this cycle. As we will see more formally in \cref{thm:edit_estimated_by_random_graph}, a result of Balogh and Martin shows that these results are more significant, because they apply to all graphs with density $p$, rather than just the random graph. We will formulate all these in terms of a quantity known as the \emph{edit distance function}.

The \emph{edit distance} is a natural metric between two graphs which counts the proportion of changes which must be made to the edge set of one in order to obtain the other. Essentially, it is a measure of how similar two graphs are. As a concept, this was first formalised by Sanfeliu and Fu \cite{SanfeliuEdit} in 1983, as a tool for pattern recognition. Here, one proposed use was as a method of using a computer to recognise lower case letters drawn by hand, and computer science has since seen many uses of this. The particular formulation we are interested in here concerns the furthest graph from some hereditary property of graphs, and was introduced independently by Alon and Stav~\cite{alon2008furthest} and Axenovich, K\'ezdy and Martin~\cite{axenovich2008editing}. 

Formally, we define the edit distance between two graphs $G$ and $G'$ on the same vertex set to be the size of the symmetric difference between their edge sets as a fraction of the total number of possible edges, that is, if $\size{V(G)}= \size{V(G')} = n$, then
\begin{equation*}
    \dist(G,G') = \frac{\size{E(G) \Delta E(G')}}{\binom{n}{2}}.
\end{equation*}
We say $\cH$ is a \emph{hereditary property of graphs} if $\cH$ is a class of graphs which is closed under taking isomorphisms and induced subgraphs. $\Forb(H)$ represents the class of all graphs $G$ which do not have $H$ as an induced subgraph. Hereditary properties can be classified in terms of their forbidden subgraphs, that is, for any hereditary property $\cH$, there is a family $\F(\cH)$ of forbidden graphs, that is,
\begin{equation*}
    \cH = \bigcap_{H \in \F(\cH)} \Forb(H).
\end{equation*}
We say a hereditary property $\cH$ is \emph{trivial} if there is an $n_0 \in \N$ such that for all $n \geq n_0$, there is no $n$-vertex graph contained in $\cH$. In other words, a class is trivial if and only if it is finite. Otherwise, we say $\cH$ is \emph{non-trivial}. 

We can extend the notion of distance between graphs to define the distance between a graph $G$ and a hereditary property $\cH$, which we define to be the minimum distance from $G$ to some graph $G'$ in $\cH$ on the same vertex set, that is,
\begin{equation*}
    \dist(G, \cH) = \min \lcb \dist(G, G') \colon G' \in \cH, \size{V(G)}= \size{V(G')} \rcb .
\end{equation*}
Problems in the area have focused on finding the maximum distance of a graph $G$ on $n$ vertices from a hereditary property $\cH$, and it was this which led to the conception of the \emph{edit distance function} by Balogh and Martin \cite{balogh2008edit}. For any $p \in [0,1]$, we define
\begin{equation}\label{eq:edit_distance_function}
    \ed_\cH(p) = \lim_{n \to \infty} \max \lcb \dist(G, \cH) \colon \size{V(G)} = n, \size{E(G)} = \floor{p \binom{n}{2}} \rcb,
\end{equation}
if this limit exists. So in other words, for any $p$, the edit distance function for a hereditary property $\cH$ tells us the furthest distance a graph of density $p$ can be from belonging to $\cH$. Balogh and Martin \cite{balogh2008edit} later generalised a result of Alon and Stav \cite{alon2008furthest} to show that the limit in \eqref{eq:edit_distance_function} does exist for all non-trivial hereditary properties $\cH$. In addition to this, they showed the following result.

%
\begin{theorem}[Balogh-Martin \cite{balogh2008edit}]\label{thm:edit_estimated_by_random_graph}
\begin{equation*}
    \ed_\cH(p) = \lim_{n \to \infty} \E \lsb \dist(G(n,p), \cH) \rsb .
\end{equation*}
\end{theorem}
That is, asymptotically, for any $p$ and hereditary property $\cH$, we can use the random graph $G(n,p)$ to estimate the edit distance function. Hence, the strategies outlined earlier for the random graph also give upper bounds on the edit distance function from $\Forb(C_h)$. Balogh and Martin \cite{balogh2008edit} also showed that the edit distance function is continuous and concave down. Methods to determine the edit distance function $\ed_\cH(p)$ make implicit use of these properties, as well as \cref{thm:edit_estimated_by_random_graph}, as we will see in the following sections. 

The edit distance function has been studied for a range of hereditary properties of the form $\cH = \Forb(H)$, such as for $H= K_r$ (see \cite{martin2013edit}), $H=K_{s,t}$ (see for example \cite{Martin_McKay_bipartite}) and more recently when $H=G(n',p')$, for some $n'$ and $p'$ which are fixed with respect to $n$ and $p$ (see \cite{martin_riasanovsky_2022}). Axenovich and Martin \cite{Axenovich_Martin_multicolor} also extended this theory into edge-coloured graphs and directed graphs, and an interesting open question raised by Martin \cite{martin2016edit} is whether this notion of the edit distance function can be extended into the setting of hypergraphs.

\subsection{Main results}\label{subsec:main_results}

In this paper, we study  the edit distance function when $\cH = \Forb(C_h^t)$ for some $h,t \in \N$. 
Here, $C_h$ is a cycle on $h$ vertices and $C_h^t$ is defined to be the graph on the same vertex set as $C_h$ such that there is an edge between two vertices of $C_h^t$ if and only if these vertices were at distance at most $t$ in $C_h$. In particular, when $t=1$, this is just the cycle on $h$ vertices. Thus, we aim to determine $\ed_{\Forb(C_h^t)}(p)$, where $\Forb(C_h^t)$ is the class of graphs which contain no $C_h^t$ as an induced subgraph. As a very natural property to consider, this question has received a lot of interest in the past. Marchant and Thomason \cite{marchant2010extremal} determined $\ed_{\Forb(C_h)}(p)$ for all $p \in [0,1]$ when $h=4$. Martin \cite{martin2013edit} explicitly determined $\text{ed}_{\text{Forb}(C_h)}(p)$ for $h \in \{5, \dots, 9\}$. In the same work, Martin determined $\ed_{\Forb(C_{10})}(p)$ for $p \in [1/7, 1]$. Peck \cite{peck2013edit} later determined $\ed_{\Forb(C_{h})}(p)$ for all $p \in [0,1]$ when $h \geq 3$ and $h$ is odd, and for $p \in [1/\ceil{h/3},1]$ when $h \geq 4$ and $h$ is even. Berikkyzy, Martin and Peck \cite{berikkyzy2019edit} generalised this result to determine $\ed_{\Forb(C_h^t)}(p)$ for all $p \in [0,1]$ when $h \geq 2t(2t+1)+1$ and $(t+1) \nmid h$. The same authors also determined $\ed_{\Forb(C_h^t)}(p)$ for $p \in [1/\ceil{h/(2t+1)},1]$ when $h \geq 2t(2t+1)+1$ and $(t+1) \mid h$. More precisely, they showed the following.

\begin{theorem}[Berikkyzy, Martin and Peck \cite{berikkyzy2019edit}]\label{pthm:berikkyzy_powers_of_cycles}
Let $t \geq 1$ and $h \geq 2t(t+1)+1$ be positive integers, and let $\cH = \emph{\text{Forb}}(C_h^t)$.
\begin{enumerate}
    \item If $(t+1) \not|\,\, h$, then for all $p \in [0,1]$, we have
    \[
        \text{\emph{ed}}_{\text{\emph{Forb}}(C_h^t)}(p) = \min\lcb \frac{p}{t+1}, \min_{r \in \{0, 1, \dots, t\}} \lcb  \frac{p(1-p)}{r+\lb \ceil{\frac{h}{2t+1}} -r-1 \rb p} \rcb \rcb.
    \]
    \item If $(t+1) \mid h$, then for all $p \in [1/\ceil{h/(2t+1)},1]$, we have
    \[
        \text{\emph{ed}}_{\text{\emph{Forb}}(C_h^t)}(p) = \min_{r \in \{0, 1, \dots, t\}} \lcb \frac{p(1-p)}{r+\lb \ceil{\frac{h}{2t+1}} -r-1 \rb p} \rcb.
    \]
\end{enumerate}
\end{theorem}

We extend on this result for small $p$ in the case when $(t+1) \mid h$ to show the following.

\begin{theorem}\label{thm:main_result}
Let $t \geq 1$ and $h \geq 4t(2t+1)$ be integers, with $(t+1) \mid h$. Let $c_0  = \floor{(\floor{h/t}+1)/3}$, $\ell_0 = \ceil{h/(2t+1)}$, and let $p_0 = t/(c_0 \ell_0 -c_0 - \ell_0 +t+1)$. Then for all $p \in [p_0, 1/\ceil{h/(2t+1)}]$, we have that 
\[
\ed_{\Forb(C_h^t)}(p) = \frac{p(1-p)}{t+\lb \ceil{\frac{h}{2t+1}} -t-1 \rb p}.
\]
\end{theorem}

Note that when $h=10$ and $t = 1$, the value of $p_0$ in the theorem above is exactly $1/7$, and so the value of $p_0$ above result matches the previously known result of Martin \cite{martin2013edit} for $h = 10$ and $t = 1$. In fact, for $t=1$, when $h \in \{10, 12\}$, we determine the edit distance function for all $p \in [0,p_0]$, as follows.

\begin{theorem}\label{thm:main_result_C_10_12}
    \begin{enumerate}
        \item\label{item:main_result_h_10} For $p \in (0,1/7)$, we have $\ed_{\Forb(C_{10})}(p) = \frac{p(1-p)}{1+2p}$.
        \item\label{item:main_result_h_12} For $p \in (0, 1/10)$, we have $\ed_{\Forb(C_{12})}(p) = \frac{p(1-p)}{1+2p}$.
    \end{enumerate}
\end{theorem}
As a consequence of \cref{thm:main_result} and \cref{thm:main_result_C_10_12} together with the work of Martin~\cite{martin2013edit} and Peck~\cite{peck2013edit}, we now know the value of $\ed_{\Forb(C_{10})}(p)$ and $\ed_{\Forb(C_{12})}(p)$ for all $p \in [0,1]$. The remainder of the paper is organised as follows. In \cref{sec:CRG}, we introduce \emph{coloured regularity graphs}, which are the key tool we use in our proof. In \cref{sec:proof}, we prove some lemmas which we then combine to prove \cref{thm:main_result}, and \cref{thm:main_result_C_10_12}. Finally, in \cref{sec:conclusion}, we discuss exactly what needs to be done to determine the edit distance function $\ed_{\Forb(C_h^t)}(p)$ for all $p \in [0,p_0]$.

\section{Coloured regularity graphs}\label{sec:CRG}

In this section we describe the concept of a coloured regularity graph, which was first introduced by Marchant and Thomason~\cite{marchant2010extremal}. The idea of this structure is that it encodes a set of rules for how to edit a graph to make it satisfy some hereditary property.

A \emph{coloured regularity graph (CRG)} is a complete graph $K$ on $k$ vertices, in which each vertex is coloured either black or white, and each edge is coloured black, white or grey. Formally, we say that the vertex set $V(K)$ is partitioned into two (possible empty) sets $\text{VB}(K)$ and $\text{VW}(K)$, of black and white vertices respectively, and that the edge set $E(K)$ is partitioned into three (possibly empty) sets $\text{EB}(K)$, $\text{EW}(K)$, and $\text{EG}(K)$ of black, white, and grey edges respectively. For ease of notation, when the CRG $K$ is clear from the context, we write $\text{VB}$ to mean $\text{VB}(K)$, and analogously write $\text{VW}$,~$\text{EB}$,~$\text{EW}$,~and~$\text{EG}$. We say a CRG $K'$ is a \emph{sub-CRG} of a CRG $K$ if $K'$ can be obtained by deleting vertices of $K$.

We say a graph $H$ \emph{embeds} in a CRG $K$, and we write $H \mapsto K$, if there exists some function $\phi \colon V(H) \longrightarrow V(K)$ which, for all $u, v \in V(K)$, satisfies the following conditions.
\begin{enumerate}[label=(\roman*)]
    \item If $uv \in E(H)$, we have either $\phi(u) = \phi(v) \in \text{VB}$, or $\phi(u)\phi(v) \in \text{EB} \cup \text{EG}$.
    \item If $uv \not\in E(H)$, we have either $\phi(u) = \phi(v) \in \text{VW}$, or $\phi(u)\phi(v) \in \text{EW} \cup \text{EG}$.
\end{enumerate}
For our purposes, due to a key property of CRGs, it is most interesting to consider all those CRGs into which $H$ does not embed. Specifically, suppose we find a CRG $K$ into which a graph $H$ does not embed. Then any graph $H'$ which contains $H$ as an induced subgraph will also not embed in $K$. So, if there is a graph $G$ which embeds into $K$, then crucially, $G$ cannot contain $H$ as an induced subgraph, implying that $G \in \text{Forb}(H)$. Thus we find a very clear relationship between the class of CRGs into which a graph $H$ does not embed, and the family $\text{Forb}(H)$.

Recall that $\F(\cH)$ is the family of all forbidden subgraphs of $\cH$. Then for any hereditary property $\cH$, we let
\begin{equation*}
    \K(\cH) = \lcb K \colon H \not\mapsto K  \text{ for all } H \in \F(\cH) \rcb .
\end{equation*}
Let $K \in \K(\cH)$ for some $\cH$, and suppose $G$ does not belong to $\cH$. We can view $K$ as a set of rules by which to edit $G$, in order for it to belong to $\cH$. We begin by partitioning $V(G)$ into $k$ parts $V_1, \dots, V_k$ such that each part $V_i$ corresponds to a distinct vertex $v_i$ of $K$. The optimal sizes of these parts are yet to be determined, but do not affect the method of editing. Indeed, we add or remove edges according the following rules.
 \begin{enumerate}[label=(\roman*),parsep=0pt]
    \item If $v_i \in \text{VB}$, we add all edges with both endpoints in $V_i$.
    \item If $v_i \in \text{VW}$, we remove all edges with both endpoints in $V_i$.
    \item If $v_iv_j \in \text{EB}$, we add all edges with one endpoint in $V_i$ and the other in $V_j$.
    \item If $v_iv_j \in \text{EW}$, we remove all edges with one endpoint in $V_i$ and the other in $V_j$.
\end{enumerate}
Let $G'$ be the graph obtained from $G$ by carrying out these edits. Then $G'$ embeds into $K$, and so by our observations above we have $G' \in \cH$. Note that CRGs corresponding to the strategies outlined in \cref{sec:introduction} would belong to $\K(\Forb(C_h))$. We will see that there will be some such CRG such that the edit distance can be determined by applying these rules to edit the graph $G(n,p)$ and then counting the expected number of edge changes which would be required. 

\subsection{Measuring the edits defined by a CRG}

Suppose we are given a CRG $K \in \K(\cH)$, and we would like to count the expected proportion of edges of $G(n,p)$ which would be changed with respect to these rules, for some fixed $p$. We will define a quadratic program $g_K(p)$ which counts exactly this quantity. In order to define $g_K(p)$, we first define a matrix $\boldM_K(p)$. We label the vertices of $K$ by $v_1, \dots v_k$, and let $\boldM_K(p)$ be a $k \times k$ matrix whose entries are given by
\begin{equation*}
    \lsb \boldM_K(p) \rsb_{ij} = \begin{cases}
                                p & \text{ if either } i = j \text{ and } v_i \in \text{VW}, \text { or } i \neq j \text{ and } v_iv_j \in \text{EW},\\
                                1-p & \text{ if either } i = j \text{ and } v_i \in \text{VB}, \text { or } i \neq j \text{ and } v_iv_j \in \text{EB},\\
                                0 & \text{ if } i \neq j \text{ and } v_iv_j \in \text{EG}.
                            \end{cases}
\end{equation*}
Then we define the quadratic program
\begin{equation}\label{eq:g_k_p}
    g_K(p) =    \begin{cases}
                    \min & \boldx^T \boldM_K(p) \boldx\\
                    \text{s.t.} & \boldx \cdot \mathbf{1} = 1\\
                         & \boldx \geq \mathbf{0}.
                \end{cases}
\end{equation}
Note that the vector $\boldx$ which has exactly one entry equal to $1$ and all other entries equal to $0$ is a feasible solution to this program. Thus, there is some optimal vector $\boldx^*$ which attains the minimum in the program above, and so $g_K(p)$ has a solution for every CRG $K$. Furthermore, for any given CRG $K$, we can easily find the value of $g_K(p)$ using the method of Lagrange multipliers.

The vector $\boldx^*$ depends on the matrix $\boldM_K(p)$, which captures information about the adjacencies in $K$. Recall that every CRG $K$ defines a partition of the random graph $G(n,p)$, where the vertices of $K$ represent parts in this partition. The vector $\boldx$ assigns a weight to every vertex of $K$. In particular, for any $v \in K$, the weight $\boldx(v)$ corresponds to the proportion of vertices of $G(n,p)$ which lie in the part corresponding to $v$, and the optimal weight vector $\boldx^*$ gives the assignment of vertices of $G(n,p)$ to parts in a way which minimises the expected proportion of edge changes required. Thus, the function $g_K(p)$ measures exactly the expected proportion of edge changes of $G(n,p)$ which the CRG $K$ defines.

The following result of Alon and Stav \cite{alon2008furthest} suggests that for any hereditary property $\cH$, we can use the function $g_K(p)$ to determine the edit distance function.
\begin{theorem}[Alon and Stav \cite{alon2008furthest}]\label{thm:alon_stav_edit_distance_equals_inf}
Let $\cH$ be a hereditary property. Then for all $p \in [0,1]$,
\[
    \text{\emph{ed}}_\cH(p) = \inf_{K \in \K(\cH)} g_K(p).
\]
\end{theorem}
Marchant and Thomason \cite{marchant2010extremal} later showed that there is in fact a CRG which attains the infimum in \cref{thm:alon_stav_edit_distance_equals_inf}, that is, they showed the following.
\begin{theorem}[Marchant and Thomason \cite{marchant2010extremal}]\label{thm:marchant_thomason_edit_distance_equals_min}
Let $\cH$ be a hereditary property. Then for all $p \in [0,1]$,
\[
    \text{\emph{ed}}_\cH(p) = \min_{K \in \K(\cH)} g_K(p).
\]
\end{theorem}
This is a key result in this area, since for any hereditary property $\cH$, the problem of determining the edit distance function is reduced to instead finding the CRG $K \in \K(\cH)$ which attains the minimum in \cref{thm:marchant_thomason_edit_distance_equals_min}. Indeed, since the strategies outlined in \cref{sec:introduction} correspond to certain CRGs $K \in \K(\Forb(C_h))$, the bounds would give an upper bound for $\ed_{\Forb(C_h)}(p)$.

An implication of \cref{thm:marchant_thomason_edit_distance_equals_min} is that for any $K \in \K(\cH)$, the function $g_K(p)$ provides an upper bound to the edit distance function. Thus, rather than examining all CRGs in $\K(\cH)$, we begin by only examining those CRGs which have all grey edges, and use these to obtain an upper bound to $\text{ed}_\cH(p)$.

We denote by $K(r,s)$ the CRG on $r+s$ vertices, which has $r$ white vertices, $s$ black vertices, and all its edges are grey. For any hereditary property $\cH$, we define the \emph{clique spectrum} to be
\begin{equation*}
    \Gamma(\cH) = \lcb (r,s) \in \Z_{\geq 0}^2 \colon H \not\mapsto K(r,s) \text{ for all } H \in \F(\cH)\rcb.
\end{equation*}
The important property of $\Gamma(\cH)$ which we will be using is its monotonicity. That is, if $(r,s) \in \Gamma(\cH)$, then for all $0 \leq r' \leq r$, $0 \leq s' \leq s$, we have $(r',s') \in \Gamma(\cH)$. This follows immediately from the definition, and gives rise to important elements of $\Gamma(\cH)$ known as \emph{extreme points}, which are pairs $(r,s) \in \Gamma(\cH)$ such that $(r+1, s), (r, s+1) \not\in \Gamma(\cH)$. We denote by $\Gamma^*(\cH)$ the set of extreme points of $\Gamma(\cH)$.

We state the following useful lemma, which allows us to observe another useful property of these grey edge CRGs. Let $K$ be a CRG. We say a sub-CRG $K'$ of $K$ is a \emph{component} if every edge leaving $K'$ is grey, that is, if for all $v \in V(K')$ and all $w \in V(K \setminus K')$, we have that $vw \in \text{EG}(K)$. So, every CRG has a `decomposition' into components, that is, a partition of the vertex set such that all edges leaving the sub-CRG induced on any part in this partition are grey. Then we can state the following lemma, a result of work by Martin \cite{martin2013edit}.
\begin{lemma}[Martin \cite{martin2013edit}]\label{lem:components}
Let $K$ be a $CRG$ with components $K^{(1)}, \dots, K^{(\ell)}$. Then
\[
    (g_K(p))^{-1} = \sum_{i=1}^{\ell} (g_{K^{(i)}}(p))^{-1}.
\]
\end{lemma}

We can now state the following useful result, which suggests that for any grey-edge CRG $K(r,s)$, it is sufficient to know $r$ and $s$ in order to calculate the value of $g_{K(r,s)}(p)$.
\begin{lemma}[Martin \cite{martin2013edit}]\label{lemma:calulate_g_of_K_r_s}
\begin{equation*}
    g_{K(r,s)}(p) = \frac{p(1-p)}{r(1-p)+sp}.
\end{equation*}
\end{lemma}

For any pair $(r, s) \in \Gamma^*(\cH)$, we have $K(r,s) \in \K(\cH)$. By minimising over all the grey edge CRGs in $\K(\cH)$, we obtain an upper bound for $\text{ed}_\cH(p)$. Formally, we define
\begin{equation*}
    \gamma_\cH(p) = \min \lcb g_{K(r,s)}(p) \colon (r,s) \in \Gamma(\cH) \rcb = \min \lcb \frac{p(1-p)}{r(1-p)+sp} \colon (r,s) \in \Gamma(\cH) \rcb.
\end{equation*}
Then $\gamma_\cH(p) \geq \text{ed}_\cH(p)$. Furthermore, suppose that $(r,s) \in \Gamma^*(\cH)$. Then for all  $0 \leq r' \leq r$, $0 \leq s' \leq s$ we have $g_{K(r,s)}(p) \leq g_{K(r',s')}(p)$. Thus when calculating $\gamma_\cH(p)$, it suffices to consider only those pairs $(r,s)$ which are extreme points of the clique spectrum. 

The advantage of this is that the value of $\gamma_\cH(p)$ is determinable for any hereditary property, and thus this upper bound is easier to calculate than directly minimising the function $g_K(p)$ over all $K \in K(\cH)$. Through the course of this paper, we may refer to a CRG which `attains $\gamma_\cH(p)$' or `attains $\text{ed}_\cH(p)$' for some value of $p$, by which we mean a CRG $K$ for which $g_K(p) = \gamma_\cH(p)$, or $g_K(p) = \text{ed}_\cH(p)$, respectively. We will also define the following special type of CRG. This was originally introduced by Peck \cite{peck2013edit}. 
\begin{definition}\label{def:candidate_CRG}
For a hereditary property $\cH$, we say that a CRG $K$ is a \emph{candidate CRG for $\cH$} if $K \in \K(\cH)$ and $g_K(p) < \gamma_\cH(p)$.
\end{definition}
If the hereditary property $\cH$ is clear from the context, we omit the phrase `for $\cH$' from \cref{def:candidate_CRG}.

\subsection{The $p$-core CRGs and symmetrisation}

As we have seen previously, the edit distance function $\text{ed}_\cH(p)$ can be determined by finding the CRG $K \in \K(\cH)$ which minimises $g_K(p)$. We say a CRG $K$ is \emph{$p$-core} if $g_K(p)<g_{K'}(p)$ for any sub-CRG $K'$ of $K$. Since we know by Theorem~\ref{thm:marchant_thomason_edit_distance_equals_min} that there exists some CRG which minimises $g_K(p)$, this definition immediately implies that there exists a $p$-core CRG which minimises $g_K(p)$. Marchant and Thomason~\cite{marchant2010extremal} identified the following useful classification of $p$-core CRGs.
\begin{theorem}[Marchant-Thomason \cite{marchant2010extremal}]\label{thm:characterisation_of_p_core_crgs}
Let $K$ be a $p$-core CRG. Then the following holds.
\begin{enumerate}
    \item If $p = 1/2$, then all edges of $K$ are grey.
    \item If $p < 1/2$, then $\text{EB} = \emptyset$ and there are no white edges incident to white vertices.
    \item If $p > 1/2$, then $\text{EW} = \emptyset$ and there are no black edges incident to black vertices.
\end{enumerate}
\end{theorem}
We consider again the quadratic program $g_K(p)$ defined in~\eqref{eq:g_k_p}. Marchant and Thomason \cite{marchant2010extremal} showed that if $K$ is a $p$-core CRG, then the optimal vector for this quadratic program is in fact unique, and moreover that this optimal vector $\boldx$ contains no zero entries.

Now the vector $\boldx$ assigns to each vertex $v \in V(K)$ a weight $\boldx(v)$. We write $d_B(v)$ for the weighted degree of $v$ along black edges, that is, $d_B(V) := \sum_{u \in V(K) : uv \in EB} \boldx(u)$, and define $d_W(v)$ and $d_G(v)$ analogously for white and grey edges.  Martin \cite{martin2013edit} found the following bounds on the quantity $d_G(v)$ using symmetrisation techniques.
\begin{lemma}[Martin \cite{martin2013edit}]\label{lem:symmetrisation}
Let $p \in (0,1/2]$ and let $K$ be a $p$-core CRG with optimal weight function $\boldx$. Then $\boldx(v) = g_K(p)/p$ for all $v \in \text{\emph{VW}}(K)$. Moreover, for all $v \in \text{\emph{VB}}(K)$, we have
\begin{equation*}
   d_G(v) = \frac{p-g_K(p)}{p}+\frac{1-2p}{p}\boldx(v).
\end{equation*}
\end{lemma}
\cref{lem:symmetrisation} can be thought of as a symmetrisation lemma,  and tells us that the weight is distributed evenly among all vertices in $\VW(K)$ by the vector $\boldx$. This gives a way of determining the value of $\boldx$ at any vertex. We can use this to state the following useful lemma.
\begin{lemma}[Martin \cite{martin2013edit}]\label{lem:upper_bound_on_black_vertex_weight}
Let $p \in (0,1/2]$ and let $K$ be a $p$-core CRG with optimal weight function $\boldx$. Then for all $v \in \text{\emph{VB}}(K)$, we have $\boldx(v) \leq g_K(p)/(1-p)$.
\end{lemma}
\section{Determining $\ed_{\Forb(C_h^t)}(p)$}\label{sec:proof}

In this section we aim to determine $\ed_{\Forb(C_h^t)}(p)$. Recall that Berikkyzy, Martin and Peck \cite{berikkyzy2019edit} determined this for all $p$ in the case when $(t+1) \nmid h$ and for $p \in [1/\ceil{h/(2t+1)}, 1]$ in the case when $(t+1) \mid h$. Therefore, we will be focusing on determining this function for $0 \leq p \leq 1/\ceil{h/(2t+1)}$ in the case when $(t+1) \mid h$.

For the remainder of this section, we let $\cH= \Forb(C_h^t)$, for integers $h$ and $t$.

\subsection{Preliminaries}

We begin by stating the value of $\gamma_\cH(p)$, which was determined by Berikkyzy, Martin and Peck \cite{berikkyzy2019edit}. 

\begin{lemma}[Berikkyzy, Martin and Peck \cite{berikkyzy2019edit}]\label{plem:gamma_function_general}
Let $t \geq 1$ and $h \geq \max\{t(t+1), 4\}$ be integers. Then for all $p \in [0,1]$, the following holds.
\begin{enumerate}
    \item\label{item:gamma_function_general_1} If $(t+1) \not| \,\, h$, then
        \[
            \gamma_\cH(p) = \min_{r \in \{0, 1, \dots, t\}} \lcb \min \lcb \frac{p}{t+1}, \frac{p(1-p)}{r(1-p)+\lb \ceil{\frac{h}{t+r+1}} -1 \rb p} \rcb \rcb.
        \]
    \item\label{item:gamma_function_general_2} If $(t+1) \mid h$ then 
        \[
            \gamma_\cH(p) = \min_{r \in \{0, 1, \dots, t\}} \lcb \frac{p(1-p)}{r(1-p)+\lb \ceil{\frac{h}{t+r+1}} -1 \rb p} \rcb.
        \]
\end{enumerate}
\end{lemma}
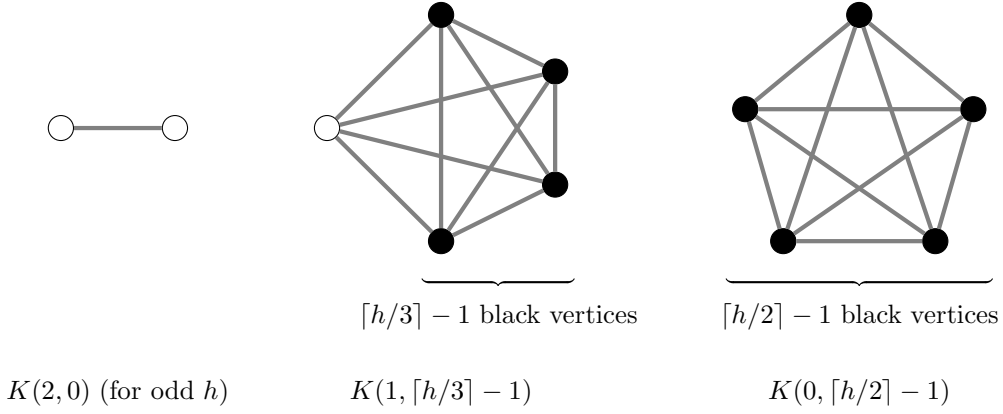
\begin{figure}
\begin{center}
\begin{tikzpicture}
\node[shape=circle,draw=black,fill=white] (A) at (3.5,2.5){};
\node[shape=circle,draw=black,fill=white] (B) at (5,2.5){};


\node[shape=circle,draw=black,fill=white] (C) at (7,2.5){};
\node[shape=circle,draw=black,fill=black] (D) at (8.5,4){};
\node[shape=circle,draw=black,fill=black] (E) at (10,3.25){};
\node[shape=circle,draw=black,fill=black] (F) at (10,1.75){};
\node[shape=circle,draw=black,fill=black] (G) at (8.5,1){};

\node[shape=circle,draw=black,fill=black] (H) at (12.5,2.75){};
\node[shape=circle,draw=black,fill=black] (I) at (14,4){};
\node[shape=circle,draw=black,fill=black] (J) at (15.5,2.75){};
\node[shape=circle,draw=black,fill=black] (K) at (15,1){};
\node[shape=circle,draw=black,fill=black] (L) at (13,1){};


\draw[ultra thick,gray] (A) -- (B);





\draw[ultra thick,gray] (C) -- (D);
\draw[ultra thick,gray] (C) -- (E);
\draw[ultra thick,gray] (C) -- (F);
\draw[ultra thick,gray] (C) -- (G);
\draw[ultra thick,gray] (D) -- (E);
\draw[ultra thick,gray] (D) -- (F);
\draw[ultra thick,gray] (D) -- (G);
\draw[ultra thick,gray] (E) -- (F);
\draw[ultra thick,gray] (E) -- (G);
\draw[ultra thick,gray] (F) -- (G);


\draw[ultra thick,gray] (H) -- (I);
\draw[ultra thick,gray] (H) -- (J);
\draw[ultra thick,gray] (H) -- (K);
\draw[ultra thick,gray] (H) -- (L);
\draw[ultra thick,gray] (I) -- (J);
\draw[ultra thick,gray] (I) -- (K);
\draw[ultra thick,gray] (I) -- (L);
\draw[ultra thick,gray] (J) -- (K);
\draw[ultra thick,gray] (J) -- (L);
\draw[ultra thick,gray] (K) -- (L);


\draw[decorate, decoration = {calligraphic brace,mirror}, thick] (8.25,0.5) --  (10.25,0.5);
\node at (9.25,0) {$\ceil{h/3}-1$ black vertices};

\draw[decorate, decoration = {calligraphic brace,mirror}, thick] (12.25,0.5)-- (15.75,0.5);
\node at (14,0) {$\ceil{h/2}-1$ black vertices};

\node (K20) at (4.25,-1) {$K(2,0)$ (for odd $h$)};
\node (K1h3) at (8.5,-1) {$K(1,\ceil{h/3}-1)$};
\node (K0h2) at (14,-1) {$K(0, \ceil{h/2}-1)$};



\end{tikzpicture}
\caption{The structures which correspond to $\Gamma^*(\Forb(C_h))$.}
\label{fig:Gamma_Forb_C_h}
\end{center}
\end{figure}

In \cref{fig:Gamma_Forb_C_h}, we see the structures which correspond to $\Gamma^*(\Forb(C_h))$ and by applying \cref{lemma:calulate_g_of_K_r_s}, we obtain the values of $g_K(p)$ for each of these CRGs $K$. This gives $\gamma_{\Forb(C_h)}(p)$, which is exactly equal to the right hand side of \cref{plem:gamma_function_general} in the case $t=1$. Observe that these are exactly the bounds given by the strategies in \cref{sec:introduction}, and in particular, the strategies in \cref{sec:introduction} correspond exactly to $\Gamma(\Forb(C_h))$. We recall the notion of a candidate CRG as in \cref{def:candidate_CRG}, that is, consider a CRG $K \in \K(\cH)$ such that $g_K(p) < \gamma_\cH(p)$ for some $p$. By \cref{pthm:berikkyzy_powers_of_cycles}, we know that in order for such a CRG to exist, we must have $(t+1) \mid h$ and $p \leq 1/\ceil{h/(2t+1)}$. We will determine some characteristics for such a candidate CRG.

By taking sub-CRGs, we may assume that $K$ is $p$-core. Thus, in particular by \cref{thm:characterisation_of_p_core_crgs}, in the case when $p \leq 1/\ceil{h/(2t+1)}$, we may assume that all edges of $K$ are grey or white, and that there are no white edges incident to white vertices. Here, we are assuming that $1/\ceil{h/(2t+1)} \leq 1/2$, which is true provided that 
\[h \geq 2(2t+1).\] 
So, in particular, we have this lower bound on $h$ throughout. Recall that $\VW$ is the set of white vertices of $K$, and $\VB$ is the set of black vertices of $K$. Let $K_B = K[\VB]$. The following lemma gives some structural conditions on $K$.

\begin{lemma}[Berikkyzy, Martin and Peck \cite{berikkyzy2019edit}]\label{lem:forbidden_cycles}
Let $p \in (0, 1/2]$\ and let $t \geq 1$ and $h \geq 2t+2$ be integers. Let $K \in \K(\cH)$ be $p$-core with $r$ white vertices.
\begin{enumerate}
    \item\label{item:lem_forbidden_cycles_1} If $r \in \{ 0, \dots, t-1 \}$ and $h \geq t(t-1)$, then $K_B$ contains no grey cycle of length in $\lcb \ceil{\frac{h}{t+r+1}}, \dots, \floor{\frac{h}{t}} \rcb$.
    \item\label{item:lem_forbidden_cycles_2} If $r = t$, then $\size{V(K_B)} \leq \ceil{\frac{h}{2t+1}}-1$.
    \item\label{item:lem_forbidden_cycles_3} If $r \geq t+1$, then $(t+1) \nmid h$ and $V(K_B)= \emptyset$.
\end{enumerate}
\end{lemma}

We will determine further characteristics of such a candidate CRG in the case that $(t+1) \mid h$ and $p < 1/\ceil{h/(2t+1)}$, with the aim of finding a contradiction and thus disproving its existence. We begin by determining the value of $\gamma_\cH(p)$ for small $p$.
\begin{lemma}\label{plem:gamma_function_for_small_p}
Let $t \geq 1$ and $h \geq t(2t+1)$ be integers with $(t+1) \mid h$. Then for $p \leq 1/\ceil{h/(2t+1)}$, we have
\[
    \gamma_\cH(p) = \frac{p(1-p)}{t(1-p) + \lb \ceil{\dfrac{h}{2t+1}}-1 \rb p}.
\]
\end{lemma}
\begin{proof}
We would like to show that if $p \leq 1/\ceil{h/(2t+1)}$ then for every $r \in \{ 0, \dots, t-1 \}$ we have
\[
    \frac{p(1-p)}{r(1-p) + \lb \ceil{\frac{h}{t+r+1}}-1 \rb p} \geq \frac{p(1-p)}{t(1-p) + \lb \ceil{\frac{h}{2t+1}}-1 \rb p},
\]
that is, the minimum in \cref{plem:gamma_function_general}~\ref{item:gamma_function_general_2} is attained when $r = t$. By rearranging, we get that the inequality above is equivalent to showing that
\begin{equation}\label{eq:showing_minimised_at_ub_on_p}
    t-r + \lb \ceil{\frac{h}{2t+1}}- \ceil{\frac{h}{t+r+1}} - t + r \rb p \geq 0.
\end{equation}
Next, note that\begin{align*}
     \ceil{\frac{h}{2t+1}}- \ceil{\frac{h}{t+r+1}} - t + r &\leq \frac{h}{2t+1} +1 -\frac{h}{t+r+1} -t+r \\&= \frac{-h(t-r)}{(2t+1)(t+r+1)} -(t-r-1) < 0.
\end{align*}
Therefore, the left hand side of \eqref{eq:showing_minimised_at_ub_on_p} is a decreasing function of $p$, so is minimised at the upper bound on $p$, that is, when $p = 1/\ceil{h/(2t+1)}$. Combined with the fact that $r \leq t-1$ and by applying the crude upper and lower bounds on the ceiling function, we have
\begin{align}
    &t-r + \lb \ceil{\frac{h}{2t+1}}- \ceil{\frac{h}{t+r+1}} - t + r \rb p \geq t - r + 1 - \frac{\ceil{\frac{h}{t+r+1}}}{\ceil{\frac{h}{2t+1}}} - \frac{t-r}{\ceil{\frac{h}{2t+1}}} \nonumber\\
    \geq \, &t - r + 1 - \frac{h(2t+1)-(2t+1)(t+r+1)(t-r-1)}{h(t+r+1)} \label{eq:want_to_be_positive}
\end{align}
Observe now that
\begin{align*}
     \frac{h}{2t+1} \geq t \geq \frac{(2t+1)(t-1)}{2t-1} \geq \frac{(t+r+1)(t-r-1)}{(t+r)(t-r)}
\end{align*}
where the first inequality holds by our assumption on $h$, the second by a straightforward manipulation, and the third since $r \in \{0,1,\dots, t-1\}$. By rearranging, we find that the right hand side of~\eqref{eq:want_to_be_positive} is non-negative, and therefore that~\eqref{eq:showing_minimised_at_ub_on_p} holds, as required. 
\end{proof}
Recall that for each pair of vertices $v, w \in V(K)$, $\boldx(v)$ is the weight assigned to the vertex $v$ by optimal vector $\boldx$ found by the quadratic program $g_K(p)$, $d_G(v)$ is the sum of the weights of all vertices incident to $v$ via a grey edge, and
$\text{deg}_G(v)$ is the number of vertices adjacent to $V$ via a grey edge. We define $d_G^W(v)$ be the sum of the weights of white vertices adjacent to $v$ via a grey edge, and let $d_G^B(v)$ be the analogous quantity for black vertices. Additionally, let $\text{deg}_G^B(v)$ be the number of black vertices adjacent to $v$ via a grey edge. We further define $d_G(v,w)$ to be the weight of the common grey neighbourhood of $v$ and $w$, and extend the definitions given above to common neighbourhoods.  Recall also that for a set of vertices $S$, $\boldx(S)$ is the sum of weights of vertices in that set. In the following proposition, we state some important properties which the vertices of a candidate CRG $K$ must satisfy. 
\begin{proposition}\label{pprop:black_vertex_properties}Let $t \geq 1$ and $h \geq t(2t+1)$ be integers with $(t+1) \mid h$. Let $K$ be a $p$-core candidate CRG, for $p \leq 1/\ceil{h/(2t+1)}$ with $g_K(p) = \gamma_\cH(p)- \eps$ for some $\eps >0$. Suppose that $K$ has $r$ white vertices, for $r \in \{ 0, \dots, t-1\}$. Then for every $v\in \VB(K)$ the following statements hold.
\begin{enumerate}
    \item\label{pitem:black_vertex_properties_overall_weighted_degree} $d_G(v) = \dfrac{t-1+\lb \ceil{\frac{h}{2t+1}}-t \rb p}{t+ \lb \ceil{\frac{h}{2t+1}}-t-1 \rb p} + \dfrac{\eps}{p} +\dfrac{1-2p}{p}\boldx(v)$.
    \item\label{pitem:black_vertex_properties_black_weighted_degree} $d_G^B(v) = \dfrac{t-r-1 + \lb \ceil{\frac{h}{2t+1}}-t+r \rb p}{t+ \lb \ceil{\frac{h}{2t+1}}-t-1 \rb p} + \dfrac{(r+1)\eps}{p} + \dfrac{1-2p}{p}\boldx(v)$.
    \item\label{pitem:black_vertex_properties_weight} $\boldx(v) \leq \dfrac{p}{t + \lb \ceil{ \frac{h}{2t+1}} - t - 1 \rb p  } - \dfrac{\eps}{1-p}$.
    \item\label{pitem:black_vertex_properties_black_unweighted_degree} $\deg_G^B(v) > \lb \ceil{\frac{h}{2t+1}} -1 \rb \lb t-r \rb$.
\end{enumerate}
\end{proposition}

\begin{proof}
Let $v \in \VB(K)$. To prove \ref{pitem:black_vertex_properties_overall_weighted_degree}, we apply \cref{lem:symmetrisation} to get that
\begin{align*}
    d_G(v) &= \frac{p-g_K(p)}{p} + \frac{1-2p}{p}\boldx(v) = 1 - \frac{\gamma_\cH(p)-\eps}{p} + \frac{1-2p}{p}\boldx(v)\\
        &= 1 - \frac{1-p}{t+ \lb \ceil{\frac{h}{2t+1}}-t-1 \rb p} + \frac{\eps}{p} +\frac{1-2p}{p}\boldx(v)= \frac{t-1+\lb \ceil{\frac{h}{2t+1}}-t \rb p}{t+ \lb \ceil{\frac{h}{2t+1}}-t-1 \rb p} + \frac{\eps}{p} +\frac{1-2p}{p}\boldx(v).
\end{align*}
Here, the third equality holds by applying \cref{plem:gamma_function_for_small_p}, since {$h \geq t(2t+1)$} and $p \leq 1/\ceil{h/(2t+1)}$.

To prove \ref{pitem:black_vertex_properties_black_weighted_degree}, we note that since $K$ is $p$-core for $p \leq 1/2$, every edge incident to a white vertex in $K$ is grey by \cref{thm:characterisation_of_p_core_crgs}. Furthermore, by \cref{lem:symmetrisation}, for each white vertex $u \in \VW(K)$, we have
\begin{equation}\label{eq:weight_of_white_vertex}
\boldx(u) = \frac{g_K(p)}{p} = \frac{1-p}{t+ \lb \ceil{\frac{h}{2t+1}}-t-1 \rb p} - \frac{\eps}{p}.
\end{equation}
Here again the second equality holds by an application of \cref{plem:gamma_function_for_small_p}. Then $d_G^B(v) = d_G(v) - d_G^W(v)=d_G(v) - \boldx(\VW)$, and we obtain the result by observing that $\boldx(\VW)$ is $r$ times the bound given in \eqref{eq:weight_of_white_vertex}, and subtracting this from the bound on $d_G(v)$ given in \ref{pitem:black_vertex_properties_overall_weighted_degree}.

To prove \ref{pitem:black_vertex_properties_weight}, we note that $\boldx(\VB) \geq \boldx(v) + d_G^B(v)$. On the other hand, we also have $\boldx(\VB) = 1 - \boldx(\VW)$. Thus, combining these gives
\begin{align*}
    \boldx(v) \leq \,\, &1 - \boldx(\VW) - d_G^B(v)\\
        =\,\, &1 - \frac{r(1-p)}{t+ \lb \ceil{\frac{h}{2t+1}}-t-1 \rb p}+\frac{r\eps}{p} - \frac{t-r-1 + \lb \ceil{\frac{h}{2t+1}}-t+r \rb p}{t+ \lb \ceil{\frac{h}{2t+1}}-t-1 \rb p} - \frac{(r+1)\eps}{p} - \frac{1-2p}{p} \boldx(v)\\
        =\,\, &\frac{1-p}{t+ \lb \ceil{\frac{h}{2t+1}}-t-1 \rb p} - \frac{\eps}{p} - \frac{1-2p}{p} \boldx(v).
\end{align*}
Rearranging and solving for $\boldx(v)$ gives the result.

Finally, to prove \ref{pitem:black_vertex_properties_black_unweighted_degree}, we have
\begin{align*}
    \text{deg}_G^B(v) &\geq \ceil{\frac{d_G^B(v)}{\max_{u \in \VB} \boldx(u)}} > \frac{t-r-1 + \lb \ceil{\frac{h}{2t+1}}-t+r \rb p}{p} \\
    &\geq \lb \ceil{\frac{h}{2t+1}}-1 \rb \lb t-r \rb,
\end{align*}
where the second inequality holds by \ref{pitem:black_vertex_properties_black_weighted_degree} and \ref{pitem:black_vertex_properties_weight} and the final inequality because $p \leq \frac{1}{\ceil{h/(2t+1)}}$.
\end{proof}

Note that when applying this result, we will usually ignore the exact value of $\eps$ and rely only on the fact that $g_K(p) < \gamma_\cH(p)$. Finally, we state the following fact, which follows immediately from Cases $2$ and $3$ in the proof of Theorem 3 in the work of Berikkyzy, Martin, and Peck \cite{berikkyzy2019edit}. This allows us to now restrict to the case when the candidate CRG contains exactly $t-1$ vertices.

\begin{lemma}\label{lem:white_vxs}
Let $p \in (0, 1/2]$\ and let $t \geq 1$ and $h \geq 2t(t+1)+1$. If $K$ is a $p$-core CRG with $g_K(p) < \gamma_\cH(p)$, that is, a candidate CRG, then $K$ has exactly $t-1$ white vertices.
\end{lemma}

\subsection{Existence of long cycles}

Before stating and proving the key lemmas in this section, we define some useful notation relating to paths and cycles. We will define these in terms of labelled paths, and remark that the corresponding definitions also apply for cycles. We say the \emph{length} of a path $P$ is the number of edges in $P$, and we denote this by $\size{P}$. If $P$ is a labelled path, say $P=v_1 \, \dots \, v_\ell$, then the \emph{successor} of a vertex $v_i$ on $P$ is the vertex $v_{i+1}$. The \emph{predecessor} of $v_i$ on $P$ is the vertex $v_{i-1}$. For indices $1 \leq i < j \leq \ell$, the subpath $v_i \, P \, v_j$ is the path $v_i \, v_{i+1} \, \dots \, v_{j-1} \, v_j$, that is, the path which has initial vertex $v_i$ and final vertex $v_j$ and follows the labelling of $P$. The path $v_j \, P^{-1} \, v_i$ is the path which has initial vertex $v_j$, final vertex $v_i$ and follows the reverse of the labelling of $P$. We also recall that for a set $S \subseteq V(K)$, the value $\boldx(S)$ denotes the sum of weights of vertices in that set, that is, $\boldx(S) = \sum_{v \in S} \boldx(v)$. We now prove a lemma which shows that for sufficiently small values of $p$, every $p$-core candidate CRG which has $t-1$ white vertices contains a grey cycle of length at least $\ceil{h/2t}$ in the subgraph induced on the black vertices. Recall that $K_B = K[\VB]$.
\begin{lemma}\label{plem:long_cycle_exists}
Let $t \geq 1$, and {$h \geq 4t(2t+1) \}$} be integers with $(t+1) \mid h$. Let $K$ be a $p$-core candidate CRG for some $p \leq 1/\ceil{h/(2t+1)}$, and suppose that $K$ contains $t-1$ white vertices. Then $K_B$ contains a grey cycle of length at least $\ceil{h/2t}$.
\end{lemma}
\begin{proof}
We would like to find a grey cycle of length at least $\ceil{h/2t}$. Recall by \cref{thm:characterisation_of_p_core_crgs}
that all edges of $K_B$ are either grey or white. Assume for a contradiction that the longest grey cycle in $K_B$ has length at most $\ceil{h/2t}-1$. By \cref{pprop:black_vertex_properties}~\ref{pitem:black_vertex_properties_black_weighted_degree}, for each vertex $v \in \VB$, we have that 
\begin{align*}
    d_G^B(v) > \frac{\lb \ceil{\frac{h}{2t+1}}-1\rb p}{t+\lb \ceil{\frac{h}{2t+1}}-t-1\rb p} + \frac{1-2p}{p}\boldx(v).
\end{align*}
Let $P$ be a longest grey path in $K_B$. Let $P = v_1 \ldots v_\ell$ so that $v_1$ and $v_\ell$ are the endvertices of $P$ and $v_{i}v_{i+1}$ are edges of $P$ for each $i \in [\ell-1]$. Note that by the maximality of $P$, each grey neighbour of $v$ in $\VB$ must lie on $P$. Let $Q = \{v_2, \ldots, v_m\}$ where $m$ is the largest index in $[\ell]$ such that $m \leq \ceil{h/2t}$ (so $Q \cup \{v_1\}$ is the set of the first $m+1$ vertices of $P$). The aim is to show that $v_1$ must have a grey neighbour in $\VB$ which lies outside $Q$, and this will give us a cycle of length at least $\ceil{h/2t}$. So we assume that $\size{Q} = \ceil{h/2t}-2$, that is $Q = \{v_2, \ldots, v_{\ceil{h/2t}-1}\}$. 

Consider $b \in \{2, \ldots, \ceil{h/2t}-1\}$ such that $v_1 v_b$ is an edge. Then by applying a rotation, we can find a grey path $P_b =  v_{b-1} P^{-1} v_1 v_b P v_\ell$ which has initial vertex $v_{b-1}$ and has the same length as $P$. By the maximality of $P$, the path $P_b$ must also be a longest grey path in $K_B$ and so the grey neighbourhood in $\VB$ of $v_{b-1}$ must also lie entirely on $P_b$. Furthermore, note that the set of vertices $Q \cup \{v_1\}$ are still the first $\ceil{h/2t}-1$ vertices of the path $P_b$. Therefore, if $v_{b-1}$ had a grey neighbour on $P$ outside $Q \cup \{v_1\}$, we would find a cycle of length at least $\ceil{h/2t}$, as required. So, we may assume that the grey neighbourhood of $v_{b-1}$ in $\VB$ lies entirely in the set $Q \cup \{v_1\}$. In particular, the set of the first $\ceil{h/(2t)}-1$ vertices of $P$ is the same as the set of the first $\ceil{h/2t}-1$ vertices of $P'$, albeit that they have been rearranged.

We know by \cref{pprop:black_vertex_properties}~\ref{pitem:black_vertex_properties_black_unweighted_degree} that $v_1$ has at least $\ceil{h/(2t+1)}$ grey neighbours in $\VB$. Say $v_1$ has exactly $\ceil{h/(2t+1)}+y$ grey neighbours in $\VB$ for some $y \in \{0, \ldots, \ceil{h/2t}-\ceil{h/(2t+1)}-2\}$. Let $Q'$ be the set of vertices in $Q$ which are predecessors in $P$ of grey neighbours of $v_{1}$, that is, the set of vertices $v_j \in V(P)$ for which $v_1 v_{j+1}$ is a grey edge and $j \in Q$. Then $\size{Q'}=\ceil{h/(2t+1)}+y-1$ (since $v_1$ itself is a predecessor of $v_2$ on $P$, but $v_1$ does not lie in $Q$). Let $Q'' = Q \setminus Q'$. Then $\size{Q''} = \ceil{h/2t}-\ceil{h/(2t+1)}-y-1$. Now by the upper bound on the weight of a black vertex given in \cref{pprop:black_vertex_properties}~\ref{pitem:black_vertex_properties_weight}, we obtain that
\begin{equation*}
 \boldx(Q) \leq \frac{\lb \ceil{\frac{h}{2t}}-2 \rb p}{t + \lb \ceil{\frac{h}{2t+1}}-t-1 \rb p}.
\end{equation*}
On the other hand, since $\boldx(Q) \geq d_G^B(v_1)$, we have 
\begin{equation*}
    \boldx(Q) > \frac{\lb \ceil{\frac{h}{2t+1}}-1 \rb p }{t + \lb \ceil{\frac{h}{2t+1}}-t-1 \rb p} + \frac{1-2p}{p} \boldx(v_1).
\end{equation*}
Furthermore, again by the upper bound on the weight of a black vertex, we have 
\begin{equation*}
    \boldx(Q'') \leq \frac{\lb \ceil{\frac{h}{2t}}-\ceil{\frac{h}{2t+1}}-y-1 \rb p}{t + \lb \ceil{\frac{h}{2t+1}}-t-1 \rb p}.
\end{equation*}
Therefore, since $\boldx(Q') = \boldx(Q) - \boldx(Q'')$, we have
\begin{equation*}
    \boldx(Q') > \frac{\lb 2\ceil{\frac{h}{2t+1}}-\ceil{\frac{h}{2t}} +y \rb p }{t + \lb \ceil{\frac{h}{2t+1}}-t-1 \rb p} + \frac{1-2p}{p} \boldx(v_1).
\end{equation*}
Therefore, by averaging over all vertices in $Q'$, we find that there exists some vertex $u \in Q'$ such that
\begin{equation*}
    \boldx(u) > \frac{\lb 2\ceil{\frac{h}{2t+1}}-\ceil{\frac{h}{2t}} +y \rb p }{\lb \ceil{\frac{h}{2t+1}}+y-1 \rb \lb t + \lb \ceil{\frac{h}{2t+1}}-t-1 \rb p \rb} + \frac{1-2p}{\lb \ceil{\frac{h}{2t+1}}+y-1 \rb p} \boldx(v_1).
\end{equation*}
Since {$h \geq 4t(2t+1)$}, we have that 
\begin{equation*}
    \boldx(u) > \frac{\lb 2\ceil{\frac{h}{2t+1}}-\ceil{\frac{h}{2t}} \rb p }{\lb \ceil{\frac{h}{2t+1}}-1 \rb \lb t + \lb \ceil{\frac{h}{2t+1}}-t-1 \rb p \rb} + \frac{1-2p}{\lb \ceil{\frac{h}{2t}}-3 \rb p} \boldx(v_1).
\end{equation*}
This bound holds because $(x + y)/(x'+y) \geq x/x'$ whenever $x' \geq x$ and $y \geq 0$. So, if we let $x = 2 \ceil{h/(2t+1)}-\ceil{h/2t}$ and $x' = \ceil{h/(2t+1)}-1$, then it suffices to have $x \leq x'$, and in order for this to be true, it suffices to have $h \geq 4t(2t+1)$. Now note that
\begin{equation}\label{eq:upper_bound_on_qv1}
    \boldx(Q \cup \{ v_1 \}) \leq \frac{\lb \ceil{\frac{h}{2t}}-2 \rb p}{t + \lb \ceil{\frac{h}{2t+1}}-t-1 \rb p} + \boldx(v_1).
\end{equation}
On the other hand, since $u \in Q'$, it is the predecessor on $P$ to some grey neighbour of $v_1$ and so as we have seen already, the entire grey neighbourhood of $u$ in $K_B$ must lie in $Q \cup \{v_1\}$. Therefore, 
\begin{align}
    &\,\,\boldx(Q \cup \{v_1\}) > \frac{\lb \ceil{\frac{h}{2t+1}}-1 \rb p}{t + \lb \ceil{\frac{h}{2t+1}}-t-1 \rb p} + \frac{1-p}{p}\boldx(u) \nonumber\\
&> \frac{\lb \ceil{\frac{h}{2t+1}}-1 \rb p}{t + \lb \ceil{\frac{h}{2t+1}}-t-1 \rb p} + \frac{\lb 2\ceil{\frac{h}{2t+1}}-\ceil{\frac{h}{2t}} \rb (1-p) }{\lb \ceil{\frac{h}{2t+1}}-1 \rb \lb t + \lb \ceil{\frac{h}{2t+1}}-t-1 \rb p \rb} + \frac{(1-2p)(1-p)}{\lb \ceil{\frac{h}{2t}}-3 \rb p^2} \boldx(v_1) \nonumber\\
&= \frac{2\ceil{\frac{h}{2t+1}}-\ceil{\frac{h}{2t}} + \lb \ceil{\frac{h}{2t+1}}^2 -4\ceil{\frac{h}{2t+1}}+\ceil{\frac{h}{2t}}+1 \rb p}{\lb \ceil{\frac{h}{2t+1}}-1 \rb \lb t + \lb \ceil{\frac{h}{2t+1}}-t-1 \rb p \rb} + \frac{(1-2p)(1-p)}{\lb \ceil{\frac{h}{2t}}-3 \rb p^2} \boldx(v_1). \label{eq:bound_in_terms_of_v_1}
\end{align}
Now we prove the following claim.
\begin{claim}\label{pclaim:long_cycles}
If $t \geq 1$, $h \geq 4(2t+1)$ and $p \leq 1/\ceil{h/(2t+1)}$, we have
\[
    \frac{(1-2p)(1-p)}{\lb \ceil{\frac{h}{2t}}-3 \rb p^2} \geq 1.
\]
\end{claim}
\begin{proof}[Proof of \cref{pclaim:long_cycles}]
We would like to show that when $p \leq 1/\ceil{h/(2t+1)}$, we have $(1-2p)(1-p)/p^2 \geq \ceil{h/2t}-3$. It suffices to show that whenever $p \leq (2t+1)/h$, we have
\[
    \frac{1-3p}{p^2} \geq \frac{h}{2t}-2.
\]
Since the left side of the equation above is decreasing in $p$ in the interval $(0,2/3)$, it suffices to show that this holds for the upper bound on $p$, that is, it suffices to show
\begin{equation}\label{eq:bound_for_claim3.6}
1-3\lb \frac{2t+1}{h} \rb \geq \lb \frac{2t+1}{h} \rb^2 \lb  \frac{h}{2t}-2 \rb.
\end{equation}
Now suppose that $h = k(2t+1)$ for some $k > 0$. Then in order for the inequality in \eqref{eq:bound_for_claim3.6} to hold, we need
\[
1-\frac{3}{k} \geq \frac{1}{k^2} \lb \frac{k(2t+1)}{2t} -2 \rb.
\]
or equivalently, we need $k^2 -k(4+1/2t)+2 \geq 0$. In particular, since $t \geq 1$, it suffices to have $k^2 - 4.5 k +2 \geq 0$. The left side of this inequality is a positive quadratic with roots at $k=0.5$ and $k=4$, so in particular, this holds whenever $k \geq 4$. Therefore, in particular, it suffices to have $h \geq 4(2t+1)$. 
\end{proof}
Note that we have {$h \geq 4t(2t+1) \geq 4(2t+1)$} and $p \leq 1/\ceil{h/(2t+1)}$. Therefore, by applying \cref{pclaim:long_cycles} to \eqref{eq:bound_in_terms_of_v_1}, we get 
\begin{equation*}
    \boldx(Q \cup \{v_1\}) > \frac{2\ceil{\frac{h}{2t+1}}-\ceil{\frac{h}{2t}} + \lb \ceil{\frac{h}{2t+1}}^2 -4\ceil{\frac{h}{2t+1}}+\ceil{\frac{h}{2t}}+1 \rb p}{\lb \ceil{\frac{h}{2t+1}}-1 \rb \lb t + \lb \ceil{\frac{h}{2t+1}}-t-1 \rb p \rb} + \boldx(v_1).
\end{equation*}
Combining this with the upper bound on $\boldx(Q \cup \{v_1\})$ given in \eqref{eq:upper_bound_on_qv1}, and rearranging, we get that
\begin{equation}\label{eq:intermediate_bound_on_p}
p > \frac{2\ceil{\frac{h}{2t+1}} -\ceil{\frac{h}{2t}}}{\ceil{\frac{h}{2t}}\ceil{\frac{h}{2t+1}} - \ceil{\frac{h}{2t+1}}^2 -2\ceil{\frac{h}{2t}} + 2 \ceil{\frac{h}{2t+1}}+1}.
\end{equation}
However, when {$h \geq 4t(2t+1)$}, the inequality above contradicts the assumption that $p \leq 1/\ceil{h/(2t+1)}$. For $t\geq 2$, we can show this by the fact that $m \leq \ceil{m} \leq m+1$. Indeed, since $p \leq 1/\ceil{h/(2t+1)}$, we know that the right side of the inequality~\eqref{eq:intermediate_bound_on_p} above must also be at most $1/\ceil{h/(2t+1)}$. Rearranging this and substituting the appropriate upper or lower bound for $m = \frac{h}{2t+1}$ or $m=\frac{h}{2t}$, and we obtain an inequality in terms of $h$ and $t$. Specifically, we show that if $p \leq 1/\ceil{h/(2t+1)}$ and also satisfies \eqref{eq:intermediate_bound_on_p}, then
\begin{equation}\label{eq:intermediate_bound_2}
\frac{2h^2(t-1)-2h(2t-1)(2t+1)^2-10(2t+1)^2t}{h^2+h(2t+1)(8t^2+4t-1)+8t(2t+1)^2}<0.
\end{equation}
As the denominator of this is always non-negative for $t \geq 1$, we require $2h^2(t-1)-2h(2t-1)(2t+1)^2-10(2t+1)^2t<0$. For $t \geq 2$, this is a positive quadratic with minimum at $h= (2t-1)(2t+1)^2/(2(t-1))$. For $t \geq 2$, this is smaller than $h=4t(2t+1)$, so for $h \geq 4t(2t+1)$, the function is increasing. By showing that the left side of \ref{eq:intermediate_bound_2} is positive at $h=4t(2t+1)$, we obtain the necessary contradiction. This does not work for $t=1$, however, since this crude upper and lower bound on $\ceil{m}$ is not quite good enough. Instead, we split into $3$ cases depending on the value of $h$ mod $3$ (that is, we replace $h$ by one of $3k$, $3k+1$, or $(3k+2)$) and then obtain a contradiction as before $p$ satisfying both $p \leq 1/\ceil{h/(2t+1)}$ and also \cref{eq:intermediate_bound_on_p}. 
We conclude that the black-vertex subgraph $K_B$ must indeed contain a grey cycle of length at least $\ceil{h/2t}$.
\end{proof}

\subsection{Using long cycles to find shorter cycles}

In this section, we prove two key lemmas. First, we show that given a graph satisfying a particular condition on the degree, we can use `long' cycles to find shorter cycles in the graph.

\begin{lemma}\label{lem:length_reduction}
Let $G$ be a graph, $m \geq 5$ a positive integer. Suppose that in every set of $\floor{m/3}$ vertices of $G$, there are some two which have a common neighbour. If $G$ contains a cycle of length at least $m$, then it also contains a cycle of length between~$\ceil{m/2}$~and~$m-1$.
\end{lemma}
\begin{proof}
Let $C$ be a shortest cycle in $G$ among all cycles of length at least $m$. The idea of the proof is as follows. We carefully choose a set of $\floor{m/3}$ vertices of $C$, and our hypothesis then implies that two of these vertices share a common neighbour. We use this common neighbour to find a new cycle $C'$ which has length strictly shorter than $C$ but still has length at least $\ceil{m/2}$. By minimality of $C$, we must have that $C'$ has length at most $m-1$, thus proving the result. It remains only to choose the set of vertices, and show that we can always find a new cycle $C'$.

We begin by labelling the vertices of $C$ by $u_1, \dots, u_\ell$, so that $u_i u_{i+1}$ is an edge for each $i \in [\ell]$ (indices taken with addition modulo $\ell$). Let $M$ be the subset of vertices of $C$ given by 
\[
    M = \lcb u_{3i-2} \colon i \in \lcb 1, \dots, \floor{\frac{m}{3}} \rcb \rcb.
\]
Since $M$ contains exactly $\floor{m/3}$ vertices, we know by hypothesis that there exist two distinct vertices $u_i, u_j \in M$ which have a common neighbour $v$ in $G$. Without loss of generality, we may assume $i <j$. By choice of $M$, the paths $u_i \, C \, u_j$ and $u_j \, C \, u_i$ have length at least 3. Hence, the path $u_i \, v \, u_j$ requires at least one edge which does not belong to the cycle $C$.

We have two cases to consider. Recall that in each case, we aim to find a new cycle $C'$ such that $\ceil{m/2} \leq \size{C'} \leq \size{C}-1$.

\begin{enumerate}[label=\textbf{Case \arabic*:},wide,labelindent=0pt,parsep=0pt]
\item \textbf{$\boldsymbol{v}$ lies outside the cycle $\boldsymbol{C}$.} Now note that $3 \leq j-i \leq m-3$. If $j-i \geq \ceil{m/2}-2$ then let $C' = u_i \, C \, u_j \, v \, u_i$. Then $\size{C'} = (j-i)+2$. Hence,
\[
\ceil{\frac{m}{2}} = \ceil{\frac{m}{2}}-2 + 2 \leq \size{C'} \leq \size{C}-3+2 = \size{C}-1,
\]
as required. On the other hand, if $j-i \leq \ceil{m/2}-2$, then let $C' = u_i \, v \, u_j \, C \, u_i$. Then $\size{C'} = \size{C}-(j-i)+2 $ and
\[
\ceil{\frac{m}{2}} \leq  m -\ceil{\frac{m}{2}} + 2 + 2 \leq  \size{C'} \leq \size{C}-1,
\]
as required. 

\item \textbf{$\boldsymbol{v}$ lies on the cycle $\boldsymbol{C}$.} Since both the paths $u_i \, C \, u_j$ and $u_j \, C \, u_i$ have length at least 3, and the vertex $v$ must lie on one of these paths, we may assume without loss of generality that both of the paths $P_1 = u_i \, C \, v$ and $P_2 = v \, C \, u_i$ have length at length at least 2. In particular, the edge $u_i \, v$ is a chord on $C$. Thus, either $P_1$ or $P_2$ have length at least $\ceil{\size{C}/2}$, say this is $P_1$ without loss of generality, and let $C' = P_1 \, u_i$. Then 
\[
    \ceil{\frac{m}{2}} \leq \frac{\size{C}}{2}+1 \leq \size{C'} \leq \size{C}-1,
\]
as required.\qedhere
\end{enumerate}
\end{proof}

The following lemma gives a condition on the common neighbourhoods of vertices in a weighted graph $K$. For each vertex $v \in V(K)$, let $N_G(v)$ be the set of grey neighbours of $v$ in $V(K)$, that is,
\[
N_G(v) = \{u \in V(K) \colon uv \in \EG\}.
\]

\begin{lemma}\label{lem:any_set_has_pair_with_intersecting_neighbourhood}
Let $K$ be a CRG with all black vertices and all edges grey or white. Let $m \geq 2$ be an integer. If $d_G(v) >1/m$ for all $v \in V(K)$, then in each subset $M \subseteq V(K)$ of size at least $m$ there exist at least 2 vertices which have a common neighbour in their grey neighbourhood.
\end{lemma}
\begin{proof}
Consider a set $M \subseteq V(K)$ with $\size{M} \geq m$. Suppose that the neighbourhoods of vertices in $M$ are pairwise-disjoint. Then we have
\[
1 \geq \sum_{w \in \bigcup_{v \in M} N_G(v)} \boldx(w) = \sum_{v \in M} \sum_{w \in N_G(v)} \boldx(w) = \sum_{v \in M} {d_G(v)} > 1,
\]
a contradiction. Here, the first inequality holds since the sum of weights of all vertices in $K$ is exactly 1, the second equality holds by the assumption that every pair of vertices have a disjoint neighbourhood and the final inequality holds because each term of the sum has size greater than $1/m$ by assumption.
\end{proof}

\subsection{Proof of Theorem~\ref{thm:main_result}}

It remains only to combine the lemmas we have seen in this section to prove our main result. 

\begin{proof}[Proof of \cref{thm:main_result}]
Let {$t \geq 1$ and $h \geq 4t(2t+1)$}. Recall that $c_0  = \floor{(\floor{h/t}+1)/3}$, $\ell_0 = \ceil{h/(2t+1)}$, and $p_0 = t/(c_0 \ell_0 -c_0 - \ell_0 +t+1)$, with $p \in [p_0, 1/\ceil{h/(2t+1)}]$.

Assume for a contradiction that there exists a $p$-core candidate CRG $K$. Suppose that $K$ has $r$ white vertices. By \cref{lem:white_vxs}, we $r = t-1$.
By \cref{plem:long_cycle_exists}, we know that $K_B$ contains a grey cycle of length at least $\ceil{h/(2t)}$. By \cref{lem:forbidden_cycles}, $K_B$ cannot contain a grey cycle of length $\ell$ for any $\ell \in \{ \ceil{h/(2t)}, \ldots, \floor{h/t} \}$. Combining these facts together, we may assume that $K_B$ contains a grey cycle of length at least $\floor{h/t}+1$. 

By definition we have $\boldx(\VB) \leq 1$. Furthermore, by \cref{pprop:black_vertex_properties}~\ref{pitem:black_vertex_properties_black_weighted_degree} as there are exactly $t-1$ white vertices, we know that for each vertex $v \in \VB$,
\begin{equation}\label{eq:weight_of_black_neighbourhood_in_grey}
    d_G^B(v) \geq \frac{\lb \ceil{\frac{h}{2t+1}}-1 \rb p}{t+ \lb \ceil{\frac{h}{2t+1}}-t-1 \rb p} + \frac{1-2p}{p}\boldx(v).
\end{equation}

\begin{claim}\label{claim:where_we_use_p_0}
    When $p \geq p_0$, we have that $d_G^B(v) > 1/c_0$.
\end{claim}
\begin{proof}[Proof of \cref{claim:where_we_use_p_0}]
    It suffices to show that when $p \geq p_0$,
    \begin{equation*}
        \frac{\lb \ceil{\frac{h}{2t+1}}-1 \rb p}{t+ \lb \ceil{\frac{h}{2t+1}}-t-1 \rb p} > \frac{1}{c_0},
    \end{equation*}
    as by \eqref{eq:weight_of_black_neighbourhood_in_grey} this implies the claim. This is equivalent to showing
    \begin{equation}\label{eq:equivalent_to_claim}
        \lb \lb \ceil{\frac{h}{2t+1}}-1 \rb c_0  -  \ceil{\frac{h}{2t+1}}+t+1 \rb p > t.
    \end{equation}
    Now note that 
    \begin{align*}
        \lb \ceil{\frac{h}{2t+1}}-1 \rb \floor{\frac{\floor{\frac{h}{t}}+1 }{3}}  -  \ceil{\frac{h}{2t+1}}+t+1 
        &\geq \lb {\frac{h}{2t+1}}-1 \rb \lb {\frac{h}{3t}} - 1 \rb  -  {\frac{h}{2t+1}}-1+t+1\\
        &= \frac{h^2}{3t(2t+1)} - \frac{h}{3t} -\frac{h}{2t+1} +1 +t \\
        &= \frac{h^2 - h(5t+1)+3t(t+1)(2t+1)}{3t(2t+1)}.
    \end{align*}
    This is positive whenever $h^2 - h(5t+1)+3t(t+1)(2t+1) \geq 0$. Indeed, this has roots at
    \begin{equation*}
        t_0^\pm = \frac{(5t+1) \pm \sqrt{(5t-1)^2-12t(t+1)(2t+1)}}{2}.
    \end{equation*}
    For $t \geq 4$, the quantity $(5t-1)^2-12t(t+1)(2t+1)$ is negative, and so this is a positive quadratic with no roots, and so is always positive. On the other hand, for $t = 1$, this is a positive quadratic with larger root between $5$ and $5.5$, and so for $h \geq 4t(2t+1) \geq 6$, the equation is positive. For $t = 2$, this has larger root between $8$ and $8.5$ and so for $h \geq 4t(2t+1) \geq 8$, the equation is positive. Finally, for $t=3$, this has larger root at $9$, and therefore, for $h \geq 4t(2t+1) \geq 9$, the equation is always positive. 

    In particular, when $t \geq 1$ and $h \geq 4t(2t+1)$, the inequality $h^2 - h(5t+1)+3t(t+1)(2t+1) \geq 0$ holds. Therefore, we can rearrange \eqref{eq:equivalent_to_claim} and get that
    \begin{equation}
        p \geq \frac{t}{\lb \ceil{\frac{h}{2t+1}}-1 \rb \floor{\frac{\floor{\frac{h}{t}}+1 }{3}}  -  \ceil{\frac{h}{2t+1}}+t+1 } = p_0.
    \end{equation}
\end{proof}

So by \cref{claim:where_we_use_p_0}, since $p \geq p_0$, each vertex $v \in \VB(K)$ satisfies $d_G^B(v) > \boldx(\VB)/ \ceil{(\ceil{h/t}+1)/3}$. Therefore by \cref{lem:any_set_has_pair_with_intersecting_neighbourhood}, in each set of $\ceil{\frac{\floor{h/t}+1}{3}}$ vertices of $K_B$, there are some two which have a common grey neighbour in $K_B$. Thus, we can apply \cref{lem:length_reduction} to show that $K_B$ must contain a grey cycle of length $\ell$ for some $\ell \in \{ \ceil{(\floor{h/t}+1)/2}, \dots, \floor{h/t}\}$. Note that $\ceil{(\floor{h/t}+1)/2} \geq \ceil{h/(2t)}$. Indeed, if $h = x + y t$ for $x \in \{0, \ldots, t-1\}$, then
\begin{align*}
    \ceil{\frac{\floor{\frac{h}{t}}+1}{2}}-\ceil{\frac{h}{2t}} = \ceil{\frac{\floor{\frac{x+yt}{t}}+1}{2}}-\ceil{\frac{x+yt}{2t}} = \ceil{\frac{y}{2}+\frac{1}{2}}-\ceil{\frac{y}{2} + \frac{x}{2t}} \geq 0.
\end{align*}
Here, the second equality holds because of the definition of the floor function and the final inequality holds because $x/t < 1$. Thus, we have found a cycle of length in the range forbidden by \cref{lem:forbidden_cycles}, giving a contradiction to the assumption that such a $K$ exists, and therefore, $\ed_{\Forb(C_h^t)}(p) = \gamma_{\Forb(C_h^t)}(p)$ for $p \in [p_0, 1/\ceil{h/(2t+1)}]$, concluding the proof.
\end{proof}

\subsection{Proof of \cref{thm:main_result_C_10_12}}

In order to prove \cref{thm:main_result_C_10_12}, we first need to introduce the following definition. Let $K$ be a CRG. Then for any vertices $u, v \in V(K)$, we say that $\dist_G(u,v)$ is the length of the shortest grey-edge path between $u$ and $v$. For any $i \in \N$, we define $N^i_G(v)= \{u \in V(K) \colon \dist_G(u,v) \leq i\}$.

\begin{proof}
    Let $h \in \{10, 12\}$ and let $p_{10} = 1/7$ and $p_{12} = 1/10$. Let $p \in (0, p_h)$ and suppose there is a $p$-core candidate CRG $K$ for $\Forb(C_h)$. We will find a contradiction to show that such a CRG cannot exist, which therefore implies that $\ed_{\Forb(C_h)}(p) = \gamma_{\Forb(C_h)}(p) = \frac{p(1-p)}{1+(\ceil{h/3}-2)p}$. By \cref{thm:characterisation_of_p_core_crgs}, we may assume that all vertices of $K$ are black, and all edges are either grey or white. Thus whenever it is not explicitly stated in this proof, we consider grey edges to be edges and white edges to be non-edges. By \cref{plem:long_cycle_exists}, we know that $K$ contains a grey cycle of length at least $h/2$. Let $C$ be the shortest grey cycle of length at least $h/2$ in $K$. Then by \cref{lem:forbidden_cycles}, we know that $K$ cannot contain a grey cycle which has length in the range $\{h/2, \ldots, h\}$. Therefore, we may assume that $C$ has length at least $h+1$. Let $u$ and $u'$ be (as close as possible to) diametrically opposite vertices on $C$ (that is, $u C u'$ and $u' C u$ both have length at least $\floor{\size{C}/2}$). Note in particular that $\floor{\size{C}/2} \geq h/2$. 
    
    Suppose first that $N_G^2(u) \cap N_G^2(u') \neq \emptyset$, that is, there is some $v$ which lies in $N_G^2(u) \cap N_G^2(u')$. Note that no neighbour of $u$ (other than the vertices adjacent to $u$ on $C$) can lie on $C$, and similarly for $u'$. If this were the case, we would find a chord in $C$, which would give a shorter cycle of length at least $\floor{\size{C}/2} \geq h/2$, a contradiction. Similarly, if the vertex $v$ lies on the cycle, then it would lie on some path between $u$ and $u'$ on the cycle. Then the path from $u$ to $v$ outside the cycle has length $2$ and the path from $u'$ to $v$ outside the cycle has length $2$. Combining these paths with the path from $u$ to $u'$ on the cycle which does not include $v$, we would get a cycle which is strictly shorter than $C$ but has length at least $h/2$. Therefore, $v$ also lies outside the cycle. This would give a path of length either $2$ or $4$ between $u$ and $u'$ outside the cycle. Taking this together with a path from $u$ to $u'$ on $C$ again gives a shorter cycle than $C$ which still has length at least $h/2$, a contradiction. Thus, we may assume that $N_G^2(u) \cap N_G^2(u') = \emptyset$. We know that $\boldx(N^2_G(u)) + \boldx(N^2_G(u')) \leq 1$. Thus, without loss of generality, assume that $\boldx(N^2_G(u)) \leq 1/2$. 

    Now for any vertex $w \in N^2(u)$, we say that the multiplicity of $w$ is the number of vertices $v \in N_G(u)$ such that $w \in N_G(v)$. By \cref{pprop:black_vertex_properties}, we know that for any $v \in V(K)$, $d_G(v) \geq \frac{3p}{1+2p}+\frac{1-2p}{p} \boldx(v)$, and that $\boldx(v) \leq \frac{p}{1+2p}$. Thus, in particular, $\size{N_G(v)} \geq 3$. The remainder of the proof differs depending on whether \ref{item:main_result_h_10} $h=10$, or \ref{item:main_result_h_12} $h=12$.

    {For \ref{item:main_result_h_10}:} Note that $K$ cannot contain cycles of length in the range $\{5,\ldots, 10\}$. Suppose that there is some vertex $v \in N_G(u)$ such that there are 2 vertices $w, w' \in N_G(v)\setminus \{u\}$ which have multiplicity at least 3. Note that this implies there exists some $v' \in N_G(u)$ such that $v' w \in \EG(K)$. If there is some $v'' \in N_G(u) \setminus \{v, v', w\}$ such that $w' v'' \in \EG(K)$, then $u v' w v w' v'' u$ is a cycle of length $6$, a contradiction. Therefore, there can be no such $v''$, and so $ w \in N_G(u)$ and $v', w \in N_G(w')$. But then $u v w w' v' u$ is a cycle of length $5$, a contradiction. Therefore, for each vertex $v \in N_G(u)$, there can be at most one vertex $b(v)$ which has multiplicity at least $3$. Then
      \begin{align*}
        \sum_{v \in N_G(u)} \sum_{w \in N_G(v) \setminus \{u, b(v)\}} \boldx(w) &= \sum_{v\in N_G(u)} \lb \lb \sum_{w \in N_G(v)} \boldx(w) \rb - \boldx(u)\rb = \sum_{v \in N_G(u)} \lb d_G(v) - \boldx(u) - \boldx(b(v)) \rb\\
        &> \sum_{v\in N_G(u)} \lb \frac{3p}{1+2p} + \frac{1-2p}{p} \boldx(v) - \boldx(u) - \boldx(b(v))\rb \\
        &\geq \sum_{v \in N_G(u)} \frac{p}{1+2p} + \frac{1-2p}{p} \boldx(v)\\
        &> \frac{p}{1+2p} \size{N_G(u)} + \frac{1-2p}{p}\lb \frac{3p}{1+2p} + \frac{1-2p}{p} \boldx(u) \rb \\
        & \geq \frac{3-3p}{1+2p} + \frac{(1-2p)^2}{p^2} \boldx(u) \geq \frac{3-\frac{3}{7}}{1+\frac{2}{7}}=\frac{18}{9}=2.
    \end{align*}
Here, the second equality comes from the definition of $d_G(v)$, the next inequality by the lower bound on $d_G(v)$, the next by using the upper bound on the weight of a vertex, the next by the bound on $d_G(u)$ and the next by the fact that $\size{N_G(u)} \geq 3$. However, for each vertex $w$ which has a path of length exactly 2 to $u$, the number of times $\boldx(w)$ is counted in this sum is exactly the multiplicity of $w$, and therefore in particular, as we are ignoring any vertices with multiplicity at least $3$, each vertex in $N^2_G(u)$ can appear at most twice in this sum. Thus, 
\[
\sum_{v \in N_G(u)} \sum_{w \in N_G(v) \setminus \{u, b(v)\}} \boldx(w) \leq 2 \cdot \boldx(N_G^2(u)) \leq 1,
\]
a contradiction. Therefore, $K$ cannot exist. 

For \ref{item:main_result_h_12}: The proof follows similarly. Note that $K$ cannot contain any cycle of length in the range $\{6, \ldots, 12\}$. Suppose there is some vertex $v \in N_G(u)$ such that there are two vertices $w$ and $w'$ in $N_G(v)$ with multiplicity at least $4$. Then in particular, there are two vertices $v'$ and $v''$ in $N_G(u)$ such that $v' w$ and $v'' w' \in \EG(K)$. Then $u v' w v w' v'' u$ gives a cycle of length $6$, a contradiction. Therefore, for each vertex $v \in N_G(u)$, there can be at most one vertex $b'(v)$ which has multiplicity at least $4$. Then by the same calculation as before with $p \leq p_{12} = 1/10$,
\begin{align*}
        \sum_{v \in N_G(u)} \sum_{w \in N_G(v) \setminus \{u, b'(v)\}} \boldx(w) > \frac{3-3p}{1+2p} + \frac{(1-2p)^2}{p^2} \boldx(u) \geq \frac{3-\frac{3}{10}}{1+\frac{2}{10}}=\frac{27}{12} > 2.
    \end{align*}
However, by the same reasoning as before, each vertex in $N^2_G(u)$ can only appear at most 3 times in the sum above, and so
\[
\sum_{v \in N_G(u)} \sum_{w \in N_G(v) \setminus \{u, b'(v)\}} \boldx(w) \leq 3 \cdot \boldx(N_G^2(u)) \leq 3/2,
\]
a contradiction. Thus again, such a CRG $K$ cannot exist.
\end{proof}
\section{Concluding remarks}\label{sec:conclusion}

We conclude by discussing what needs to be done to determine the value of $\ed_{\Forb(C_h^t)}(p)$ for $p \in [0,p_0]$. For notational simplicity we discuss this only for the case $t=1$ (i.e. for cycles), but we expect that any successful method for $t=1$ would extend to the cases $t\geq 2$, and improvements for $t=1$ would be interesting in their own right.

For $t=1$, consider the following setup.
\begin{setup}\label{setup:t_equals_1}
Let $2 \mid h$, let $n \geq h$ and let $p \in (0,p_0]$. Suppose $G$ is a graph on $n$ vertices, and $\boldx \colon V(G) \rightarrow [0,1]$ is such that $\sum_{v \in V(G)} \boldx(v) =1$. Suppose further that for any $v \in V(G)$, 
\[
\sum_{u \in N(v)} \boldx(u) \geq \frac{\lb\ceil{\frac{h}{3}}-1\rb p}{1+\lb \ceil{\frac{h}{3}} -2 \rb p} + \frac{1-2p}{p} \boldx(v).
\]
\end{setup}
Suppose under \cref{setup:t_equals_1} we can prove that $G$ must contain a cycle with length in the range $\{h/2, \ldots, h\}$. Then this would prove that $\ed_{\Forb(C_h)}(p) = \gamma_{\Forb(C_h)}(p)$ for all $p \in (0,p_0]$. On the other hand, suppose that there is some way to construct $G$ satisfying \cref{setup:t_equals_1} without containing any cycles in the range $\{h/2, \ldots, h\}$. Then this would correspond to a CRG $K$ with vertex set $\VB(K) = V(G)$ and $\VW(K) = \emptyset$ such that 
\begin{itemize}
    \item $\EW(K) = \emptyset$,
    \item $u v \in \EB(K)$ if and only if $u v \notin E(G)$, and
    \item $u v \in \EG(K)$ if and only if $u v \in E(G)$.
\end{itemize}
Furthermore, this CRG would satisfy that $g_K(p) < \gamma_{\Forb(C_h)}(p)$, and therefore would give another upper bound on the value of the edit distance function for $p \in (0, p_0]$. The authors would conjecture that the first of these holds, that is, that $\ed_{\Forb(C_h)}(p) = \gamma_{\Forb(C_h)}(p)$ for all $p \in (0,p_0]$, and any further improvements within this range would be interesting.

\medskip
\bibliographystyle{abbrv}

\end{document}